\documentclass[12pt]{amsart}
\usepackage{amsmath,amsthm,amssymb,amsfonts,enumerate,color}
\usepackage{amsmath, amsthm, amsfonts, amssymb, color}
\usepackage{mathrsfs}
\usepackage{amsmath, amsthm, amsfonts, amssymb, color}
\usepackage{mathrsfs}
\usepackage{amsfonts, amsmath}
\usepackage{amsmath,amstext,amsthm,amssymb,amsxtra}
\usepackage{txfonts} 
\allowdisplaybreaks
 \usepackage{pgf,tikz}
\usepackage{stmaryrd}
\begin{document}

\newcommand{\norm}[1]{\left\Vert#1\right\Vert}
\newcommand{\abs}[1]{\left\vert#1\right\vert}
\newcommand{\set}[1]{\left\{#1\right\}}
\newcommand{\Real}{\mathbb{R}}
\newcommand{\RR}{\mathbb{R}^n}
\newcommand{\supp}{\operatorname{supp}}
\newcommand{\card}{\operatorname{card}}
\renewcommand{\L}{\mathcal{L}}
\renewcommand{\P}{\mathcal{P}}
\newcommand{\T}{\mathcal{T}}
\newcommand{\A}{\mathbb{A}}
\newcommand{\K}{\mathcal{K}}
\renewcommand{\S}{\mathcal{S}}
\newcommand{\blue}[1]{\textcolor{blue}{#1}}
\newcommand{\red}[1]{\textcolor{red}{#1}}
\newcommand{\Id}{\operatorname{I}}

\newtheorem{thm}{Theorem}[section]
\newtheorem{prop}[thm]{Proposition}
\newtheorem{cor}[thm]{Corollary}
\newtheorem{lem}[thm]{Lemma}
\newtheorem{lemma}[thm]{Lemma}
\newtheorem{exams}[thm]{Examples}
\theoremstyle{definition}
\newtheorem{defn}[thm]{Definition}
\newtheorem{rem}[thm]{Remark}

\numberwithin{equation}{section}

\title[Temperature  of Schr\"odinger operators with initial data in  Morrey spaces]
{Characterization of temperatures associated to Schr\"odinger operators with initial data in Morrey spaces}

 \author[Q. Huang and C. Zhang]{Qiang Huang\  and\ Chao Zhang}

 \address{Department of Mathematics \\ Zhejiang Normal University\\
         Jinhua 321004, PR China}
 \email{huangqiang0704@163.com}

 \address{School of Statistics and Mathematics \\
             Zhejiang Gongshang University \\
             Hangzhou 310018, PR China}
 \email{zaoyangzhangchao@163.com}


\begin{abstract}  Let $\L$ be a Schr\"odinger operator of the form $\L=-\Delta+V$ acting on $L^2(\mathbb R^n)$ where the nonnegative
potential $V$ belongs to  the reverse H\"older class    $B_q$ for some $q\ge n.$ Let $L^{p,\lambda}(\mathbb{R}^{n})$, $0\le \lambda<n$ denote
the Morrey space on $\mathbb{R}^{n}$.
In this paper, we will
show that a function  $f\in L^{2,\lambda}(\mathbb{R}^{n})$ is the trace of the solution of
${\mathbb L}u:=u_{t}+\L u=0,  u(x,0)= f(x),$
 where $u$ satisfies  a Carleson-type condition
\begin{eqnarray*}
 \sup_{x_B, r_B} r_B^{-\lambda}\int_0^{r_B^2}\int_{B(x_B, r_B)}  |\nabla   u(x,t)|^2 {dx dt }  \leq C <\infty.
\end{eqnarray*}
Conversely, this Carleson-type   condition characterizes  all the ${\mathbb L}$-carolic functions whose traces belong to
the Morrey space $L^{2,\lambda}(\mathbb{R}^{n})$ for all $0\le \lambda<n$.
This result extends the  analogous characterization found by Fabes and Neri in \cite{FN1}
for   the classical BMO space  of John and Nirenberg.

 \end{abstract}

\maketitle

\section{Introduction and statement of the main result}
\setcounter{equation}{0}

In Harmonic Analysis, to study a (suitable) function $f(x)$ on $\mathbb R^n$ is to consider a harmonic function
on $\mathbb R^{n+1}_{+}$ which has the boundary value as $f(x)$. A standard choice for such a harmonic
function is the Poisson
integral $e^{-t\sqrt{-\Delta}} f(x)$ and one recovers $f(x)$ when letting $t \rightarrow 0^{+}$, where $\Delta=\sum_{i=1}^n\partial_{x_i}^2$ is the Laplace operator.
 In other words,
one obtains $u(x,t) = e^{-t\sqrt{-\Delta}} f(x)$ as the solution of the equation
\begin{equation*}
\left\{
\begin{aligned}
\partial_{tt}u +\Delta u=0&,\ \ \ \  \ \ \ \  x\in \RR, \ t>0, \\
u(x,0)= f(x)&, \ \ \ \ \ \ \ \ x\in \RR.\\
\end{aligned}
\right.
\end{equation*}
This approach is intimately related to
the study of singular integrals. In \cite{SW}, the authors studied the classical case $f \in L^p(\mathbb R^n)$, $1 \le p \le \infty$.

It is well known that the BMO space, i.e. the space
of functions of bounded mean oscillation, is natural substitution to study singular integral at the end-point space $L^{ \infty}(\RR)$.
A celebrated   theorem of  Fefferman and Stein \cite{FS}   states    that
 a BMO function   is the trace of the solution of
 $\partial_{tt}u +\Delta u=0,  u(x,0)= f(x),$
 whenever $u$ satisfies
\begin{eqnarray*}\label{ee1.1}
 \sup_{x_B, r_B} r_B^{-n}\int_0^{r_B}\int_{B(x_B, r_B)}|t\nabla  u(x,t)|^2 {dx dt\over t }  \leq C<\infty,
\end{eqnarray*}
where $\nabla=(\nabla_x, \partial_t)=(\partial_{1},...,\partial_{n}, \partial_t).$  Conversely,
   Fabes, Johnson and Neri \cite{FJN} showed that  condition  above
   characterizes  all the harmonic functions whose traces are in ${\rm BMO}(\RR)$ in 1976.
  The study of this topic has been widely
extended to more general operators such as elliptic operators and Schr\"odinger operators (instead of the Laplacian),  for more general initial data spaces and for domains other than $\mathbb R^n$ such as Lipschitz domains. For these generalizations,   see \cite{DKP, DYZ, FN1, FN,  HMM, MSTZ2, Song}.

In \cite{FN1}, Fabes and Neri further generalized  the above characterization to caloric functions (temperature), that is the authors proved that a BMO function $f$ is the trace of the solution of

\begin{equation}\label{heateq}
\left\{
\begin{aligned}
\partial_{t}u-\Delta u=0&,\ \ \ \  \ \ \ \  x\in \RR, \ t>0, \\
u(x,0)= f(x)&, \ \ \ \ \ \ \ \ x\in \RR,\\
\end{aligned}
\right.
\end{equation}
 whenever $u$ satisfies
\begin{eqnarray} \label{e1.1}
 \sup_{x_B, r_B} r_B^{-n}\int_0^{r_B^2}\int_{B(x_B, r_B)}  |\nabla_x  u(x,t)|^2 {dx dt }  \leq C<\infty,
\end{eqnarray}
  and, conversely,  the  condition \eqref{e1.1}
   characterizes  all the carolic functions whose traces are in ${\rm BMO}(\RR)$.
The authors in \cite{JX} explored more informations, related to harmonic functions and carolic functions, about this subject.

In this paper, we consider a similar characterization in Moerry space $L^{p,\lambda}(\mathbb{R}^{n})$. It is known that $L^{p,0}(\mathbb{R}^{n})=L^{p}(\mathbb{R}^{n})$ and $L^{p,\lambda}(\mathbb{R}^{n})
=\mathcal{C}^{p,\lambda}(\mathbb{R}^{n})/\mathbb{C}$ for $0\leq\lambda<n$, where $\mathcal{C}^{p,\lambda}(\mathbb{R}^{n})$
denote the Campanato space. When $\lambda=n$, $\mathcal{C}^{p,n}(\mathbb{R}^{n})=\rm BMO(\mathbb{R}^{n})$.
So, Morrey spaces were proposed to be intermediate function spaces between $L^{p}$
space and BMO space. For more information about Morrey spaces, see \cite{YSY}.
The main aim of this article is to study a similar characterization  to \eqref{heateq}  for the Schr\"odinger operator
with some conditions on its potentials and boundary values in Morrey spaces.
To be precise, let us consider  the Schr\"odinger  operator
\begin{equation}\label{e1.2}
\L  =-\Delta  +V(x) \ \ \ {\rm on} \ \ L^2(\RR), \ \ \ \  n\geq 3.
\end{equation}
We assume that $V$ is a nonnegative  potential, not identically zero and that
 $V\in B_q$  for  some $q\geq n/2$,
which by definition means that   $V\in L^{q}_{\rm loc}(\RR), V\geq 0$, and
there exists a  constant $C>0$ such that   the reverse H\"older inequality
\begin{equation}\label{e1.3}
\left(\frac{1}{\abs{B}}\int_BV(y)^q~dy\right)^{1/q}\leq\frac{C}{\abs{B}}\int_BV(y)~dy,
\end{equation}
holds for all   balls $B$ in $\RR.$

The operator   $\L$ is a self-adjoint
operator on $L^2(\RR)$. Hence $\L$ generates the $\L$-heat semigroup $$\T_tf(x)=e^{-t{\L}}f(x)=\int_{\Real^n}\mathcal{H}_t(x,y)f(y)dy,\quad  f\in L^2(\RR),\ t>0.$$
From the Feynman-Kac formula, it is well-known that  the semigroup kernels ${\mathcal H}_t(x,y)$  of the operators $e^{-t{\L}}$ satisfies
\begin{eqnarray*}
0\leq {\mathcal H}_t(x,y)\leq h_t(x-y)
\end{eqnarray*}
for all $x,y\in\RR$ and $t>0$, where
$$
h_t(x)=(4\pi t)^{-\frac{n}{2}}e^{-\frac{|x|^2}{4t}}
$$
is the kernel of the classical heat semigroup
$\set{{T}_t}_{t>0}=\{e^{t\Delta}\}_{t>0}$ on $\Real^n$. For the classical heat semigroup associated with Laplacian, see \cite{St1970}.
In this article, we consider the parabolic Schr\"odinger differential operators
$${\mathbb L}=\partial_{t}+{\L} ,$$
$t>0, x\in\RR$;  see,
for instance, \cite{GJ,TH,YZ} and references therein.
For $f\in  L^p(\RR)$, $1\leq p<  \infty,$
it is well known that $u(x,t)=e^{-t{\L}}f(x), t>0, x\in\RR$, is a solution to the heat equation
\begin{eqnarray*}\label{el.4}
{\mathbb L}u=\partial_{t}u+{\L} u =0\ \ \ {\rm in }\ {\mathbb R}^{n+1}_+
\end{eqnarray*}
with the boundary data $f\in  L^p(\RR)$, $1\leq p<  \infty.$
 The equation  ${\mathbb L}u =0$ is interpreted in the weak sense via a sesquilinear form, that is,
  $u\in {W}^{1, 2}_{{\rm loc}} ( {\mathbb R}^{n+1}_+) $ is a weak solution of ${\mathbb L}u =0$   if it satisfies
$$\int_{{\mathbb R}^{n+1}_+}
{\nabla_x}u(x,t)\cdot {\nabla_x}\psi(x,t)\,dxdt-\int_{{\mathbb R}^{n+1}_+} u(x,t) \partial_t\psi(x,t)dxdt+
 \int_{{\mathbb R}^{n+1}_+} V u\psi \,dxdt=0,\ \ \ \ \forall \psi\in C_0^{1}({\mathbb R}^{n+1}_+).
 $$
 In the sequel,   we call such a function $u$  an ${\mathbb L}$-carolic function associated to the operator ${\mathbb L}$.

In \cite{YZ}, the authors proved that the conclusion gotten by E. Fabes and U. Neri in \cite{FN1} can be proved in the Schr\"odinger case. In \cite{Song}, the authors considered the results in \cite{FJN} in the case of Poisson integrals of
Schr\"odinger operators with Morrey traces. As mentioned above,  we  are interested in  deriving the characterization of the solution
 to the heat equation $ {\mathbb L}u  =0$ in ${\mathbb R}^{n+1}_+$ with boundary values in Morrey spaces.
 Recall that Morrey spaces were introduced in 1938 by C. Morrey\cite{Morrey}  to consider the regularity problems of
 solutions to PDEs. For every $1\leq p<\infty$ and $\lambda\in[0,n)$, the Morrey spaces $L^{p,\lambda}(\mathbb{R}^{n})$
 are defined as
 \begin{equation*}
 L^{p,\lambda}(\mathbb{R}^{n})=\left\{f\in L^{p}_{loc}(\mathbb{R}^{n}):\sup\limits_{x\in\mathbb{R}^{n},r>0}
 r^{-\lambda}\int_{B(x,r)}|f(y)|^{p}dy<\infty\right\}
 \end{equation*}
 This is a Banach space with respect to the norm
 \begin{equation*}
 \|f\|_{L^{p,\lambda}(\mathbb{R}^{n})}=\left(\sup\limits_{x\in\mathbb{R}^{n},r>0}
 r^{-\lambda}\int_{B(x,r)}|f(y)|^{p}dy\right)^{1/p}<\infty
 \end{equation*}
Moreover, for every $1\leq p<\infty$ and $\lambda>0$, the Campanato spaces $C^{P,\lambda}(\mathbb{R}^{n})$ are defined as
\begin{equation*}
C^{p,\lambda}(\mathbb{R}^{n})=\left\{f\in L_{loc}^{p}(\mathbb{R}^{n}):\|f\|_{C^{p,\lambda}(\mathbb{R}^{n})}<\infty\right\}
\end{equation*}
with the Campanato seminorm being given by
\begin{equation*}
\|f\|_{C^{p,\lambda}(\mathbb{R}^{n})}=\left(\sup\limits_{x\in\mathbb{R}^{n},r>0}r^{-\lambda}\int_{B(x,t)}|f(y)-f_{B(x,r)}|^{p}dy\right)^{1/p}<\infty,
\end{equation*}
where $f_{B(x,r)}$ denotes the average value of $f$ on the ball $B(x,r)$. And when $\lambda\in[0,n)$, $\mathcal{C}^{p,\lambda}(\mathbb{R}^{n})/\mathbb{C}=L^{p,\lambda}(\mathbb{R}^{n})$. Specially, when  $\lambda=0$,  $\mathcal{C}^{p,0}(\mathbb{R}^{n})/\mathbb{C}=L^{p,0}(\mathbb{R}^{n})=L^{p}(\mathbb{R}^{n}).$

Next, we introduce a new  function class on the upper half plane $\mathbb{R}_+^{n+1}$.
\begin{defn}
Suppose $V\in B_q$ for some $q\ge n$ and $0\le \lambda<n.$ We say that,  a $C^1$-functions $u(x,t)$ defined on $\Real_+^{n+1}$ belongs to the class  ${\rm TL_\L^\lambda }(\Real_+^{n+1})$, if  $u(x,t)$ is
the solution of ${\mathbb L}u=0$  in $\Real_+^{n+1} $ such that
\begin{eqnarray*}\label{e1.8}
\|u\|^2_{{\rm TL_\L^\lambda }(\Real_+^{n+1})}&=& \sup_{x_B, r_B}     r_B^{-\lambda} \int_0^{r_B^2}\int_{B(x_B, r_B)}  | \nabla u(x,t) |^2
{dx dt }  <\infty,
\end{eqnarray*}
where $\nabla=(\nabla_x, \partial_t).$
\end{defn}

 The following theorem is the main result of this article.

 \begin{thm}\label{th1.1}
 Suppose $V\in B_q$ for some $q\ge n$ and $0\le \lambda<n,$
then we have
 \begin{itemize}
\item[(1)] if $f\in L^{2,\lambda}(\mathbb{R}^{n})$, then  the function $u=e^{-t \L}f\in {\rm TL_\L^\lambda }(\Real_+^{n+1})$
 and
 $$
 \|u\|_{{\rm TL_\L^\lambda }(\Real_+^{n+1})}\leq C\|f\|_{L^{2,\lambda}(\mathbb{R}^{n})};$$

\item[(2)]    if $u\in {\rm TL_\L^\lambda }(\Real_+^{n+1})$, then    there exists some $f\in L^{2,\lambda}(\mathbb{R}^{n})$ such that $u=e^{-t \L}f$,
and
$$
\|f\|_{L^{2,\lambda}(\mathbb{R}^{n})}\leq C\|u\|_{{\rm TL_\L^\lambda }(\Real_+^{n+1})}
$$
with some constant $C>0$ independent of $u$ and $f$.

 \end{itemize}
\end{thm}

We should mention that for the Schr\"odinger operator $\L$ in \eqref{e1.2},
an important property of the $B_q$ class, proved in \cite[Lemma 3]{Ge}, assures that the condition $V\in B_q$ also implies $V\in B_{q+\epsilon}$
for some $\epsilon>0$
 and that the $B_{q+\epsilon}$ constant of $V$ is controlled in terms of the one of $B_q$ membership. This in particular implies
 $V\in L^q_{\rm loc}(\RR)$ for some $q$ strictly greater than $n/2.$ However,  in general the potential $V$ can be unbounded and does not
 belong to $L^p(\RR)$ for any $1 \le p \le \infty .$  As a model example, we could take $V(x)=|x|^2$.
 Moreover, as
  noted in \cite{Shen}, if $V$ is any nonnegative
 polynomial, then $V$ satisfies the stronger condition
 \begin{eqnarray*}
\max_{x\in B} V(x)\leq\frac{C}{\abs{B}}\int_BV(y)~dy,
\end{eqnarray*}
which implies $V\in B_q$ for every $q\in (1, \infty)$ with a uniform constant.

This article is organized as follows. In Section 2, we recall some preliminary results including
the  kernel estimates of the heat semigroup related with $\mathcal{L},$  
 and prove some lemmas and certain properties of ${\mathbb L}$-carolic functions.
In  Section 3, we will prove our main result,   Theorem~\ref{th1.1}.

Throughout the article, the letters ``$c$ " and ``$C$ " will  denote (possibly different) constants
which are independent of the essential variables.

\vskip 1cm

\section{Basic properties of the heat  semigroups of Schr\"odinger operators}
\setcounter{equation}{0}


In this section, we begin by recalling some basic properties of the nonnegative  potential $V$ under the assumption \eqref{e1.3} and
the  kernel estimates of the heat semigroup related with $\mathcal{L}$.

It follows from Lemmas 1.2 and 1.8 in \cite{Shen}  that there is a constant $C_0$ such that
for a nonnegative Schwartz class function $\varphi$ there exists a constant $C$ such that
\begin{equation*}\label{e2.2}
\int_{\RR} \varphi_t(x-y)V(y)dy\leq
\left\{
\begin{array}{lll}
Ct^{-1}\left({\sqrt{t}\over \rho(x)}\right)^{\delta}\ \ \ &{\rm for}\ t\leq \rho(x)^2,\\
C\left({\sqrt{t}\over \rho(x)}\right)^{C_0+2-n}\ \ \ &{\rm for}\ t> \rho(x)^2,
\end{array}
\right.
\end{equation*}
where $\varphi_t(x)=t^{-n/2}\varphi(x/\sqrt{t}),$
$
\delta =2-\frac{n}{q}>0,
$ and the critical radii function $\rho(x; V)=\rho(x)$ above are determined by the function
\begin{equation*}\label{e1.7}
 \rho(x)=\sup \Big{\{} r>0: \ {1\over r^{n-2}} \int_{B(x, r)} V(y)dy \leq 1 \Big{\}}.
\end{equation*}

For the heat kernel ${\mathcal H}_t(x,y)$
of the semigroup
$e^{-t\L}$, we have the following estimates.

\begin{lem}[see \cite{DGMTZ}] \label{le2.2} Suppose $V\in B_q$ for some $q> n/2.$
For every $N>0$, there exists a constant $C_N$ such that for every $ x,y\in\Real^n, t >0$,
 \begin{itemize}
\item[(i)]
\begin{equation*}
0\leq {\mathcal H}_t(x,y)\leq C_N t^{-{n\over 2}}e^{-\frac{\abs{x-y}^2}{ct}}\left(1+\frac{\sqrt{t}}{\rho(x)}
+\frac{\sqrt{t}}{\rho(y)}\right)^{-N} \ {\rm and}
\end{equation*}

\item[(ii)] 
\begin{equation*}
 \abs{\partial_t{\mathcal H}_t(x,y)}\leq C_N t^{-\frac{n+2}{2}}e^{- \frac{\abs{x-y}^2}{ct}}
 \left(1+\frac{\sqrt t}{\rho(x)}+\frac{\sqrt t}{\rho(y)}\right)^{-N} .
\end{equation*}
 \end{itemize}
\end{lem}
In fact, with the same computation as in the proof of \cite[Proposition 4]{DGMTZ}, we have
\begin{equation}\label{eq1}
 \abs{t^m\partial_t^m{\mathcal H}_t(x,y)}\leq C_N t^{-\frac{n}{2}}e^{- \frac{\abs{x-y}^2}{ct}}
 \left(1+\frac{\sqrt t}{\rho(x)}+\frac{\sqrt t}{\rho(y)}\right)^{-N} .
\end{equation}

\begin{lem}\label{le3.8}\cite[Lemma 3.8]{DYZ}
 Suppose $V\in B_q$ for some $q>n.$ Let $\displaystyle \beta=1-{n\over q}$.
  For every $N>0$, there exist    constants  $C=C_{N}>0$ and $c>0$ such that
  for all $x,y\in\RR$ and $t>0,$ the $\mathcal{L}$-Heat semigroup kernels ${\mathcal H}_t(x,y)$,   associated to $e^{-t{\L}}$,
   satisfy the following estimates:
   \begin{itemize}

\item[(i)]
\begin{eqnarray}\label{e3.11}
 | \nabla_x {\mathcal H}_t(x,y)| + | {t} \nabla_x \partial_t{\mathcal H}_t(x,y)|
 \leq C t^{-(n+1)/2}e^{-\frac{\abs{x-y}^2}{ct}}\left(1+\frac{\sqrt{t}}{\rho(x)}
+\frac{\sqrt{t}}{\rho(y)}\right)^{-N},
\end{eqnarray}
 \item[(ii)] for $|h|<|x-y|/4,$
\begin{eqnarray*}\label{e3.12}
 | \nabla_x{\mathcal H}_t(x+h,y)- \nabla_x{\mathcal H}_t(x,y)|
 \leq C\left({|h|\over \sqrt{t}}\right)^{\beta}
t^{-(n+1)/2}e^{-\frac{\abs{x-y}^2}{ct}};
\end{eqnarray*}
\item[(iii)]  there is some $\delta>1$ such that
\begin{eqnarray*}\label{e3.13}
 \big|\sqrt t \nabla_x e^{-t{\L}}(1)(x) \big|\le C \min \left\{ \left(\frac{\sqrt t}{\rho(x)}\right)^\delta , \left(\frac{\sqrt t}{\rho(x)}\right)^{-N}\right\}.
\end{eqnarray*}
\end{itemize}
\end{lem}

\smallskip

We now recall a local behavior of solutions to $\partial_{t}u+{\L} u =0$, which was proved in \cite[Lemma 3.3]{WY}, see it also in \cite[Lemma 3.2]{GJ}.
We define parabolic cubes of center $(x,t)$ and radius $r$ by $B_{r}(x,t):=\{(y,s)\in {\RR}\times\Real_+:\abs{y-x}<r,\, t-r^2<s\le t\}=B(x,r)\times(t-r^2,t]$.
And for every $(x,t), (y,s)\in \RR\times (0, \infty)$, we define the parabolic metric: $\abs{(x,t)-(y,s)}=\max\{\abs{x-y}, \abs{s-t}^{1/2}\}$.

\begin{lemma}\cite[Lemma 3.3]{WY}\label{le2.6} Suppose $0\leq V\in L^q_{\rm loc}(\RR)$ for some $q> n/2.$
  Let $u$ be a weak solution of ${\mathbb L}u=0$
in the parabolic cube $B_{r_0}(x_0,t_0)$.
Then there exists a constant $C=C_n>0$ such that

\begin{eqnarray*}
\sup_{B_{r_0/4}(x_0,t_0)}| u(x,t)| \leq C\Big({1\over r_0^{n+2}}
\int_{B_{r_0/2}(x_0,\, t_0)}| u(x,t)|^2dxdt\Big)^{1/2}.
\end{eqnarray*}
\end{lemma}

\medskip

 \medskip

\section{Proof of  the Main Theorem }
\setcounter{equation}{0}

In this section, we will give the proof of  Theorem \ref{th1.1}. First, we need make some preparations.

 \medskip

\begin{lem}\label{le3.1} For every  $u\in {\rm TL_\L^\lambda}(\Real_+^{n+1})$ and
 for every $k\in{\mathbb N}$, there exists a constant $C_{k,n}>0$ such that
\begin{equation*} \label{dd}
\int_{\RR}{|u(x,{1/k})|^2\over (1+|x|)^{2n}}  dx\leq C_{k,n} \|u\|^2_{{\rm TL_\L^\lambda}(\Real_+^{n+1})}<\infty,
\end{equation*}
hence $u(x, 1/k)\in L^2((1+|x|)^{-2n}dx)$. Therefore for  all $k\in{\mathbb N}$, $e^{-t{\L}}(u(\cdot, {1/k}))(x)$ exists
everywhere in ${\mathbb R}^{n+1}_+$.
\end{lem}

\begin{proof}  Since $u\in C^{1}({\mathbb R}^{n+1}_+)$, it  reduces to show that for every $k\in{\mathbb N},$
\begin{eqnarray}\label{e3.1}
\int_{|x|\geq 1} {|u(x,{1/k})- u(x/|x|, 1/k) |^2 \over (1+|x|)^{2n}} dx\leq C_{k,n}\|u\|^2_{{\rm TL_\L^\lambda}(\Real_+^{n+1})}<\infty.
 \end{eqnarray}
To do this, we write
\begin{align*}
&\hspace{-0.3cm}u(x, 1/k)- u(x/|x|, 1/k)\\&=\big[u(x, 1/k)- u(x, |x|)\big]
 +\big[u(x, |x|)-u(x/|x|, |x|)\big] +\big[u(x/|x|, |x|)-u(x/|x|, 1/k)\big].
 \end{align*}
Let
  \begin{eqnarray*}
 I=  \int_{|x|\geq 1 } {|u(x, 1/k)- u(x, |x|) |^2 \over (1+|x|)^{2n}} dx,
 \end{eqnarray*}
  \begin{eqnarray*}
 II=  \int_{|x|\geq 1 } {|u(x, |x|)-u(x/|x|, |x|) |^2 \over (1+|x|)^{2n}} dx,
 \end{eqnarray*}and
  \begin{eqnarray*}
 III=  \int_{|x|\geq 1 } {|u(x/|x|, |x|)-u(x/|x|, 1/k) |^2 \over (1+|x|)^{2n}} dx.
 \end{eqnarray*}

For $\abs{x}\ge 1$ and $t>0$, let $r^2=t/4$. We use  Lemma~\ref{le2.6} for $\partial_t u$ and Schwarz's inequality to obtain
 \begin{eqnarray}\label{e3.2}
 \big| \partial_t u(x, t)\big| &\leq& C\Big({1\over r^{n+2}} \int_{t-r^2}^t\int_{B(x, r)}
 | \partial_s u(y, s) |^2 {dyds } \Big)^{1/2}\nonumber\\
  &\leq& C\Big({1\over t^{\frac{n+2}{2}}} \int_{t-r^2}^t\int_{B(x, \sqrt t/2)}
 | \partial_s u(y, s) |^2 {dyds } \Big)^{1/2}\nonumber\\
  &\leq& C t^{-\frac{1}{2}-\frac{n-\lambda}{4}}\left({1\over \abs{B(x, \sqrt t)}^{\frac{\lambda}{n}}} \int_{0}^{t}\int_{B(x,\sqrt t)}
  |  \partial_s u(y, s) |^2 {dyds } \right)^{1/2}
  \nonumber\\
  &\leq&  C  t^{-\frac{1}{2}-\frac{n-\lambda}{4}}\|u\|_{{\rm TL_\L^\lambda}(\Real_+^{n+1})},
  \end{eqnarray}
which gives
 \begin{eqnarray*}\label{e3.3}
 | u(x, 1/k)-u(x, |x|) |  =  \Big| \int_{1/k}^{|x|}   \partial_t u(x, t) dt \ \Big| \leq C(\abs{x}^{\frac{1}{2}-\frac{n-\lambda}{4}}-k^{-\frac{1}{2}+\frac{n-\lambda}{4}}) \|u\|_{{\rm TL_\L^\lambda}(\Real_+^{n+1})}.
 \end{eqnarray*}
It follows that
   \begin{eqnarray*}
 I+III
 &\leq& C\|u\|^2_{{\rm TL_\L^\lambda}(\Real_+^{n+1})} \int_{|x|\geq 1 } {1\over (1+|x|)^{2n}}  (|x|^{\frac{1}{2}-\frac{n-\lambda}{4}}-k^{\frac{n-\lambda}{4}-\frac{1}{2}})    dx  \\
 &\leq& C(k,n)  \|u\|^2_{{\rm TL_\L^\lambda}(\Real_+^{n+1})}.
  \end{eqnarray*}

For the term $II,$  we  have that for any $x\in \RR,$
 $$
  u(x, |x|)- u(x/|x|, |x|)  =\int_1^{|x|}  D_r u(r\omega, |x|)  dr, \ \ \ \ x=|x| \omega.
 $$
 Let $B=B(0, 1)$  and  $2^mB=B(0, 2^m)$.
Note that for every $m\in{\mathbb N}$,  we have
  \begin{align*}
   \int_{2^mB\backslash 2^{m-1}B} \left| \int_1^{|x|}\left|    D_r  u(r\omega, |x|) \right|   dr \right|^2 dx
  &=     \int_{2^{m-1}}^{2^{m}} \int_{|\omega|=1} \left| \int_1^{\rho}     D_r u(r\omega, \rho)    dr \right|^2 \rho^{n-1}  d\omega d\rho   \nonumber\\
      &\leq   2^{mn-m} \int_{2^{m-1}}^{2^{m}} \int_{|\omega|=1}\int_1^{ 2^{m}}    |  D_r u(r\omega, \rho)  |^2 dr d\omega d\rho   \nonumber\\
	  &\leq    2^{mn-m} \int_{2^{m-1}}^{2^{m}} \int_{2^mB\backslash B}   |  \nabla_y u(y, t)  |^2 |y|^{1-n} dy dt\nonumber\\
	  &\leq    2^{mn-m} \int_{2^{m-1}}^{2^{m}} \int_{2^mB}   |  \nabla_y u(y, t)  |^2  dy dt,
  \end{align*}
  which gives
  \begin{align*}
 \int_{2^mB\backslash 2^{m-1}B}   | u(x, |x|) - u(x/|x|, |x|) |^2    dx
    &\leq     C2^{mn-m+m\lambda} \left( {1\over |2^mB|^{\lambda\over n}}  \int_{0}^{2^{2m}} \int_{2^mB }    |  \nabla_y u(y, t)  |^2
	{dydt}   \right) \nonumber
	 \\
    &\leq  C  2^{(n+\lambda)m-m}\|u\|^2_{{\rm TL_\L^\lambda}(\Real_+^{n+1})}.
  \end{align*}
Therefore,
 \begin{align*}
II
  &\leq C \sum_{m=1}^{\infty}  {1\over 2^{2nm}}
   \int_{2^mB\backslash 2^{m-1}B}    | u(x, |x|) - u(x/|x|, |x|) |^2    dx\leq C\|u\|^2_{{\rm TL_\L^\lambda}(\Real_+^{n+1})}.
  \end{align*}
  Combining estimates of $I, II$ and $III$, we have obtained \eqref{e3.1}.

  Note that by Lemma~\ref{le2.2},  if
  $V\in B_q$ for some $q\geq n/2$, then  the semigroup kernels ${\mathcal H}_t(x,y)$, associated to $e^{-t{\L}}$,
  decay
faster than any power of $1/|x-y|$.
Hence,    for  all $k\in{\mathbb N}$, $e^{-t{\L}}(u(\cdot, {1/k}))(x)$ exists
everywhere in ${\mathbb R}^{n+1}_+$.
This completes the proof.
  \end{proof}

\begin{lem}\label{le3.2} For every  $u\in {\rm TL_\L^\lambda}(\Real_+^{n+1})$,
we have that for every $k\in{\mathbb N}$,
\begin{equation*}
u(x, t+{1/k})=e^{-t{\L}}\big(u(\cdot, {1/k})\big)(x), \ \ \ \ x\in\RR, \  t>0.
\end{equation*}
\end{lem}

 \begin{proof}
 The lemma can be proved by the same discussion as the proof of Lemma 3.2 in \cite{DYZ} and the proof of Lemma 3.3 in \cite{YZ}.
\end{proof}

We  recall that the classical  Carleson
measure is closely related to  the   space ${\rm BMO}(\RR)$.
In \cite{YZ}, the authors considered another similar Carleson measure which was called $2$-Carleson measure. Here, we need to consider a similar Carleson measure.
We say that a
measure $\mu$ defined on ${\mathbb R}^{n+1}_+$ is a $(2,\lambda)-$Carleson measure  if there is a positive constant
$c$ such  that for each ball $B$, with radius $r_B$, in ${\mathbb R}^{n}$,
\begin{equation}\label{e3.6}
\mu({\widehat B})\leq c|B|^{\lambda\over n},
\end{equation}
where $\displaystyle {\widehat B}=\{(x,t):x\in B, 0\le t\le r_B^2\}$ is the $2$-tent over $B$.
The smallest bound $c$  in (\ref{e3.6}) is defined to  be the norm of
$\mu$, and is denoted by
$\interleave\mu\interleave_{(2,\lambda)car}$. When $\lambda=n,$ it is coincided with the $2$-Carleson measure in \cite{YZ}.
By using this measure, we will estimate the term $\partial_te^{-t{\L}}f(x)$ in  Morrey space. Precisely, for any $k\in{\mathbb N}$, we set
$$
u_k(x,t)= u(x, t+{1/k} ).
$$
Following a similar argument as in \cite[Lemma 1.4]{FJN}, we have the following lemma.

\begin{lem}\label{le3.4}For every  $u\in {\rm TL_\L^\lambda}(\Real_+^{n+1})$,
 there exists a constant $C>0$ (depending only on $n$) such that for all  $k\in{\mathbb N},$
\begin{equation}\label{e3.5}
\sup_{x_B, r_B}   r_B^{-\lambda}\int_0^{r_B^2}\int_{B(x_B, r_B)}  | \partial_t u_k(x,t)|^2 {dx dt}
  \leq C\|u\|^2_{{\rm TL_\L^\lambda}(\Real_+^{n+1})}<\infty.
\end{equation}
\end{lem}

 \begin{proof}
Let $B=B(x_B, r_B)$. If $r_B^2\geq 1/k$, then letting $s=t+1/k$, it follows that
 \begin{align*}
 |B|^{-\frac{\lambda}{n}}\int_0^{r_B^2}\int_B  |   \partial_t u (x,t+1/k)|^2 {dx dt }
 &\leq  C  |B|^{-\frac{\lambda}{n}}\int_0^{(2r_B)^2}\int_{2B} | \partial_s u (x,s)|^2 {dx ds }\\
 &\leq C \|u\|^2_{{\rm TL_\L^\lambda}(\Real_+^{n+1})}<\infty .
\end{align*}
If $r_B^2< 1/k$, then it follows from Lemma~\ref{le2.6} for   $\partial_{t}u(x, t+{1/k})$ and
 a similar argument as in \eqref{e3.2}   that
 \begin{equation*}
 \big| \partial_{t}u(x, t+{1/k})\big| \leq C \big(t+k^{-1}\big)^{-{1\over 2}+\frac{n-\lambda}{4}} \|u\|_{{\rm TL_\L^\lambda}(\Real_+^{n+1})}.
  \end{equation*}
 Therefore,
  \begin{align*}
 |B|^{-\frac{\lambda}{n}}\int_0^{r_B^2}\int_B | \partial_t u (x,t+1/k)|^2 {dx dt }
 &\leq C  |B|^{-\frac{\lambda}{n}}\|u\|^2_{{\rm TL_\L^\lambda}(\Real_+^{n+1})}   \int_0^{r_B^2}\int_{B}  \big(t+k^{-1}\big)^{-(1+\frac{n-\lambda}{2})} {dx dt}\\
  &\leq C  \|u\|^2_{{\rm TL_\L^\lambda}(\Real_+^{n+1})} \,  \Big(k^{1+\frac{n-\lambda}{2}}r_{B}^{n-\lambda} \int_0^{r_B^2}  1 { dt}\Big)\\
 &\leq C \|u\|^2_{{\rm TL_\L^\lambda}(\Real_+^{n+1})}<\infty
\end{align*}
since
$r_B^2< 1/k.$

By taking the supremum over all balls $B\subset\Real^n,$ we complete the proof of \eqref{e3.5}.
\end{proof}

Letting $f_k(x)=u(x, 1/k), k\in{\mathbb N}$, it follows from Lemma \ref{le3.2}     that
$$
u_k(x,t)=e^{-t{\L}}f_k(x), \ \ \ x\in\RR, \ t>0.
$$
And it follows from
Lemma~\ref{le3.4} that
\begin{equation*}
\sup_{x_B, r_B}
 r_B^{-\lambda}\int_0^{r_B^2}\int_{B(x_B, r_B)} |  \partial_t  e^{-t \L}   f_{k}(x)|^2 {dxdt } \leq C\|u\|^2_{{\rm TL_\L^\lambda}(\Real_+^{n+1})}.
\end{equation*}

\begin{lem} \label{le3.5} For every  $u\in {\rm TL_\L^\lambda}(\Real_+^{n+1})$,
there exists a constant $C>0$ independent of $k$ such that
\begin{equation*}
\|f_k\|_{\rm L^{2,\lambda}(\RR)}\leq C\|u\|_{{\rm TL_\L^\lambda}(\Real_+^{n+1})}<\infty,\ \  \hbox{ for any } k\in \mathbb N.
\end{equation*}
Hence for  all $k\in{\mathbb N}$, $f_k$ is uniformly bounded in  ${\rm L^{2,\lambda}(\RR)}$.
\end{lem}

To prove  Lemma~\ref{le3.5}, we need to  establish the following Lemmas~\ref{le3.6} and \ref{le3.7}.

Given a function
$f\in L^2((1+|x|)^{-2n}dx)$ and an $L^2\subset L^{2n\over n+2}$ function  $g$  supported on a ball $B=B(x_B, r_B)$, for
 any $ (x,t)\in {\mathbb R}^{n+1}_+$, set
\begin{eqnarray}\label{e3.7}
 F(x,t)= t \partial_t e^{-t \L}   f(x)\ \ \ {\rm and}\ \ \
G(x,t)= t\partial_t e^{-t \L}  (I-e^{-r^2_B \L} )g(x).
\end{eqnarray}


\begin{lem} \label{le3.6}
   Suppose $f, g, F, G$ are as in (\ref{e3.7}).
 If $f$   satisfies
\begin{equation*}
\interleave \mu_{\nabla_t, f}\interleave^2_{(2,\lambda)car}=\sup_{x_B, r_B}
 r_B^{-\lambda}\int_0^{r_B^2}\int_{B(x_B, r_B)} |  \partial_t  e^{-t \L}   f(x)|^2 {dxdt }
 <\infty,
\end{equation*}
  then there exists a constant $C>0$  such that
\begin{eqnarray}\label{e3.14}
\int_{{\mathbb R}^{n+1}_+}
|F(x,t) G(x,t)|{dxdt\over t}\leq C |B|^{\frac{\lambda}{2n}}\interleave \mu_{\nabla_t, f}\interleave_{(2,\lambda)car}
\|g\|_{{ L}^{2n\over n+2}(B)}.
\label{e3.8}
\end{eqnarray}
\end{lem}

 \begin{proof}
To prove (\ref{e3.14}), let us consider  the square functions
${\mathcal S}(f)$ and ${\mathcal G}(f)$ given  by
$$
{\mathcal S}(f)(x)=\Big(\int_0^{\infty}
|t\partial_t e^{-t \L}f(x)|^2{dt}\Big)^{1/2},
$$
$$
{\mathcal G}(f)(x)=\Big(\int_0^{\infty}
|\partial_t e^{-t \L}f(x)|^2{dt}\Big)^{1/2}.
$$
By the standard spectral theory as in \cite{DGMTZ}, we have the following identities:
\begin{equation}\label{eqsquares}
\norm{{\mathcal S}(f)}_2= {1\over 2}\norm{\L^{-1/2}f}_2,
\end{equation}
and
\begin{equation}\label{eqsquare}
\norm{{\mathcal G}(f)}_2=\frac{\sqrt 2}{2}\norm{\L^{1/2}f}_2.
\end{equation}
In fact, let us denote by
$dE(\lambda)$ the spectral resolution of the operator $\L$.
Since $\displaystyle e^{-t\L}=\int_0^\infty e^{-t\lambda}~dE(\lambda)$,
we have
$$t\partial_t e^{-t\L}=\int_0^\infty t\lambda
e^{-t\lambda}~dE(\lambda).$$ Then, for all $f\in L^2(\Real^n)$, we
have
\begin{align*}\label{equ6}
    \|{t\partial_t e^{-t\L} f(x)\|}_{L^2(\Real^{n+1}_+,{dxdt})}^2
    &=\int_0^\infty\int_{\Real^n}|t\partial_te^{-t\L}f(x)|^2~dx~{dt}
    =\int_0^\infty\left\langle(t\partial_te^{-t\L})^2f,f
         \right\rangle_{L^2(\Real^n)}{dt}\nonumber\\
   &=\int_0^\infty\int_0^\infty t^{2}\lambda^{2}e^{-2t\lambda}
   ~{dt}~dE_{f,f}(\lambda)=\frac{1}{4}\norm{\L^{-1/2}f}_{L^2(\Real^n)}^2,
\end{align*}
which gives the proof of \eqref{eqsquares}. And the proof of \eqref{eqsquare} is similar.

Given a ball $B=B(x_B, r_B)\subset{\mathbb R}^n$ with radius $r_B$, we  put
$$
T(B)=\{(x,t)\in{\mathbb R}^{n+1}_+: x\in B, \ 0<t<r_B^2\}.
$$
We then write
\begin{align*}
\int_{{\mathbb R}^{n+1}_+}
&\big| F(x,t) G(x,t)\big|{dxdt\over t} \\&=\int_{T(2B)} \big|F(x,t)
G(x,t)\big|{dxdt\over t}+\sum_{k=2}^{\infty}\int_{T(2^{k}B)\backslash T(2^{k-1}B) }
\big|F(x,t) G(x,t)\big|{dxdt\over t}\\
&={\rm A_1} + \sum_{k=2}^{\infty} {\rm A_k}.
\end{align*}
Using the H\"older inequality,  \eqref{eqsquares} and the $L^2-L^{2n\over n+2}$ boundedness of fractional integral operator $\L^{-{1\over 2 }}$ ,   we obtain
\begin{align*}
 {\rm A_1}
&\leq \Big\|\Big\{\int_0^{(2r_B)^2} |  \partial_t e^{-t \L}   f(x)  |^2{dt}\Big{\}}^{1/2}\Big\|_{L^2(2B)}
\|{\mathcal S} ({
{I}}- e^{-r^2_B \L} )g\|_{{ L}^2(\RR)}\\
&\leq C r_B^{\lambda \over 2}\interleave \mu_{\nabla_t, f}\interleave_{(2,\lambda)car}
\| \L^{-{1\over 2}}({
{I}}- e^{-r_B^2 \L} )g\|_{{ L}^2(\RR)}\\
&\leq C r_B^{\lambda \over 2}\interleave \mu_{\nabla_t, f}\interleave_{(2,\lambda)car}\| g\|_{{ L}^{2n\over n+2}(B)}.
\end{align*}

Let us estimate ${\rm A_k}$ for $k=2,3, \cdots.$
Observe that
\begin{align*}
{\rm A}_k
&\leq \Big\|\Big\{\int_0^{(2^kr_B)^2}\big|\partial_t e^{-t \L} f(x)\big|^2{dt}\Big\}^{1/2}\Big\|_{ L^2(2^kB)}\\
&\quad \quad \times \Big\|
\Big\{\int_0^{(2^kr_B)^2}\big|t\partial_t e^{-t \L}  (I-e^{-r^2_B \L} )g(x)\chi_{T(2^{k}B)\backslash T(2^{k-1}B) }(x,t)\big|^2
{dt}\Big\}^{1/2}\Big\|_{{ L}^{2}(2^kB)}\\
&\leq C (2^kr_B)^{\lambda \over 2} \interleave \mu_{\nabla_t, f}\interleave_{(2,\lambda)car}\times  {\rm B}_k,
\end{align*}
where
\begin{eqnarray*}
{\rm B}_k =\Big{\|}
\Big\{\int_0^{(2^kr_B)^2}\big|t\partial_t e^{-t \L}  (I-e^{-r^2_B \L} )
g(x)\chi_{T(2^{k}B)\backslash T(2^{k-1}B) }(x,t)\big|^2
{dt}\Big\}^{1/2}\Big{\|}_{{ L}^{2}(2^kB)}.
\end{eqnarray*}
To estimate ${\rm B}_k,$ we set
$$
\Psi_{t,s} (\L)h(y)=(t+s)^2\Big({d^2{e^{-r \L}}\over dr^2}
\Big|_{r=t+s}  h\Big)(y).
$$
Note that
$$
 ({ {I}}-e^{-r^2_B \L} )g=\int_0^{r^2_B} \L e^{-s \L}g{{ds}}.
$$ By \eqref{eq1}, we have
\begin{align*}
 {\rm B}_k
& \leq C\Big{\|}\Big\{\int_0^{(2^kr_B)^2 }\Big| t\int_0^{r^2_B}
{1\over (t+s)^2} \Psi_{t,s} (\L)g(x)\chi_{T(2^{k}B)\backslash T(2^{k-1}B) }(x,t){ds} \Big|^2  {dt } \Big\}^{1/2}\Big{\|}_{{ L}^{2}(2^kB)}\\
&\leq C\Big{\|}\Big\{\int_0^{(2^kr_B)^2} \Big|t \int_0^{r^2_B} \int_{B(x_B,r_B)}
{1\over (t+s)^{{n\over 2}+2}}e^{-c\frac{\abs{x-y}^2} {t+s}} \\&\quad \quad \quad \quad \quad \quad \quad \quad
\quad \quad \quad \quad \quad \quad \quad  \times |g(y)|\chi_{T(2^{k}B)\backslash T(2^{k-1}B) }(x,t)
{dyds} \Big|^2  {dt} \Big\}^{1/2}\Big{\|}_{{ L}^{2}(2^kB)}.
\end{align*}
Note that for   $(x,t)\in T(2^{k}B)\backslash T(2^{k-1}B)$
and $y\in B$, we have that $|x-y|\geq 2^kr_B$.   So
\begin{align*}
{\rm B}_k &\le C\Big{\|}\Big\{\int_0^{(2^kr_B)^2} \Big|t \int_0^{r^2_B} \int_{B(x_B,r_B)}
{1\over \abs{x-y}^{n+4}}
 |g(y)|\chi_{T(2^{k}B)\backslash T(2^{k-1}B) }(x,t)
{dyds} \Big|^2  {dt } \Big\}^{1/2}\Big{\|}_{{ L}^{2}(2^kB)}\\
&\le C\norm{g}_{L^1(B)}\Big{\|}\Big\{\int_0^{(2^kr_B)^2}\Big|t \int_0^{r^2_B}
{1\over (2^k r_B)^{n+4}}
 \chi_{T(2^{k}B)\backslash T(2^{k-1}B) }(x,t)
{ds} \Big|^2  {dt } \Big\}^{1/2}\Big{\|}_{{ L}^{2}(2^kB)}\\
&\le C\norm{g}_{L^1(B)}\frac{r^2_B}{(2^k r_B)^{n+4}}\Big{\|}\int_0^{(2^kr_B)^2}
 \chi_{T(2^{k}B)\backslash T(2^{k-1}B) }(x,t)
\ \  t^2dt \Big{\|}^{1/2}_{{ L}^{1}(2^kB)}\\
&\leq  C { 2^{(-2-{n\over 2})k}r_B^{-\frac{n}{2}-2}\norm{g}_{L^1(B)}\le C2^{(-\frac{n}{2}-2)k}\norm{g}_{L^{2n\over n+2}(B)}.}
\end{align*}
Consequently,
\begin{equation*}
{\rm A}_k
\leq C 2^{-2k}  r_B^{\lambda\over 2}   \interleave \mu_{\nabla_t, f}\interleave_{(2,\lambda)car}\norm{g}_{L^{2n\over n+2}(B)},
\end{equation*}
which implies
\begin{align*}
\int_{{\mathbb R}^{n+1}_+}
|F(x,t) G(x,t)|{dxdt\over t}
&\leq Cr_B^{{\lambda  \over 2}}\interleave \mu_{\nabla_t, f}\interleave_{(2,\lambda)car}\|g\|_{{ L}^{2n\over n+2}(B)}
 +
C\sum_{k=2}^{\infty} 2^{-2k}  r_B^{{\lambda  \over 2}}
\interleave\mu_f\interleave_{(2,\lambda)car}\|g\|_{{L}^{2n\over n+2}(B)} \\
&\leq  Cr_B^{{\lambda  \over 2}}
\interleave \mu_{\nabla_t, f}\interleave_{(2,\lambda)car}\|g\|_{{L}^{2n\over n+2}(B)}
\end{align*}
as desired.
\end{proof}

\begin{lem} \label{le3.7}   Suppose $B, f, g, F, G $ are
defined as in Lemma~\ref{le3.6}. If  $\interleave \mu_{\nabla_t, f}\interleave_{(2,\lambda)car}<\infty$, then
we have the equality:
\begin{eqnarray*}
\int_{{\mathbb R}^n} f(x) ({\mathcal {I}}-e^{-r^2_B{\L}})g(x)dx={1\over 4}\int_{{\mathbb R}^{n+1}_+}
F(x,t) G(x,t){dxdt\over t}.
\end{eqnarray*}
\end{lem}

\begin{proof}
The technique of this lemma's proof has been used in lots of papers, for example \cite{DXY, DY2, DGMTZ, MSTZ}, but it is notable to state it at here for completeness.

By Lemma \ref{le3.6}, we know that $\displaystyle\int_{\Real_+^{n+1}}\abs{F(x,t)G(x,t)}{dxdt\over t}<\infty$. By dominated convergence theorem, the following integral converges absolutely and satisfies
\begin{equation*}
I=\int_{{\mathbb R}^{n+1}_+}F(x,t) G(x,t){dxdt\over t}=\lim_{\epsilon\rightarrow 0^+}\int_\epsilon^{1\over \epsilon}\int_{\RR}F(x,t)G(x,t){dxdt\over t}.
\end{equation*}
By Fubini's theorem, together with the commutative property of the semigroup $\{e^{-t\L}\}_{t>0}$, we have
\begin{eqnarray*}
\int_{\RR}F(x,t)G(x,t)dx
&=&\int_{\RR} f(y)\Big(t\partial_te^{-t\L}\Big)^2(\mathcal I-e^{-r^2_B\L})g(y)dy.
\end{eqnarray*}
 Whence,
 \begin{eqnarray*}
I&=&\lim_{\epsilon\rightarrow 0^+}\int_\epsilon^{1\over \epsilon}\int_{\RR} f(x)(t\partial_te^{-t\L})^2(\mathcal I-e^{-r^2_B\L})g(x){dx dt\over t}\\
&=&\lim_{\epsilon\rightarrow 0^+} \int_{\RR}f(x)\int_\epsilon^{1\over \epsilon}(t\partial_te^{-t\L})^2(\mathcal I-e^{-r^2_B\L})g(x){dtdx\over t}.
\end{eqnarray*}
By \cite[Lemma 7]{DGMTZ}, we can pass the limit inside the integral above. And, by a similar computation of \cite[Lemma 3.7]{MSTZ} with $\beta=1$ , we have
 \begin{eqnarray*}
I&=& \int_{\RR}f(x)\int_0^{\infty} (t\partial_te^{-t\L})^2(\mathcal I-e^{-r^2_B\L})g(x){dtdx\over t}={1\over 4}\int_{\RR}f(x)(\mathcal I-e^{-r^2_B\L})g(x)dx.
\end{eqnarray*}
This completes the proof.
\end{proof}

Now, we are in a position to prove  Lemma~\ref{le3.5}
\begin{proof}[Proof of Lemma~\ref{le3.5}]
First, we note an equivalent  characterization of $ {\rm L^{2,\lambda}}(\RR)$ that
   $ f\in{\rm L^{2,\lambda}}(\RR)$ if and only if  $f\in L^2((1+|x|)^{-(n+\epsilon)}dx)$ and
  \begin{equation*} \label{e3.9}
  \sup_B \Big( |B|^{-\frac{\lambda}{n}}\int_{B}|f(x)- e^{-r^2_B{\L}}f(x) |^2dx\Big)^{1/2}\leq C<\infty.
  \end{equation*}
This has been proved in \cite[Proposition 6.11]{DY2}
  (see  also \cite{DY1, HLMMY, Song}).

 Now if $\|u\|_{{\rm TL_\L^\lambda}(\Real_+^{n+1})}<\infty$, then it follows from Lemma \ref{le3.1} that
 $$
 \int_{\Real^{n}} {|f_k(x)|^2\over 1+|x|^{2n}} dx\leq C_k <\infty.
 $$
 Given an $L^2$ function $g$ supported on a ball $B=B(x_B, r_B)$, it follows by  Lemma~\ref{le3.7}  that
 we have
\begin{align*}
  \int_{{\mathbb R}^n}
f_k(x) (I-e^{-r^2_B \L} )g(x)dx
 = {1\over 4}\int_{{\mathbb R}^{n+1}_+}  t \partial_t e^{-t \L}   f_k(x)  \
t \partial_t e^{-t \L}  (\mathcal I-e^{-r^2_B \L} )g(x) {dxdt\over t}.
\end{align*}
By Lemmas~\ref{le3.4} and ~\ref{le3.6},
\begin{align*}
  \abs{\int_{{\mathbb R}^n}
f_k(x) (\mathcal I-e^{-r^2_B\L} )g(x)dx}  &\leq C|B|^{\frac{\lambda}{2n}}
 \interleave \mu_{\nabla_t,  f_k}\interleave_{(2,\lambda)car} \|g\|_{{L}^{2n\over n+2}(B)}\\
 &\leq C|B|^{\frac{\lambda}{2n}}\|u\|_{{\rm TL_\L^\lambda}(\Real_+^{n+1})} \|g\|_{{L}^{2n\over n+2}(B)}.
\end{align*}
  Then the duality argument
for ${L}^2$ shows that
\begin{align*}
\Big(|B|^{-\frac{\lambda}{n}}\int_B |f_k(x)-e^{-r^2_B \L}f_k(x)|^2dx\Big)^{1/2}
&=|B|^{-\frac{\lambda}{2n}}\sup\limits_{\|g\|_{{ L}^{2}(B)}\leq 1}\Big|\int_{{\mathbb R}^n}
(\mathcal I-e^{-r^2_B \L})f_k(x)g(x)dx\Big|\nonumber\\
&\le |B|^{-\frac{\lambda}{2n}}\sup\limits_{\|g\|_{{ L}^{2n\over n+2}(B)}\leq 1}\Big|\int_{{\mathbb R}^n}
f_k(x)\big(\mathcal I-e^{-r^2_B \L}\big)g(x)dx\Big|\nonumber\\
&\leq C
\|u\|_{{\rm TL_\L^\lambda}(\Real_+^{n+1})},
\end{align*}
for some $C>0$   independent of $k.$

 It then follows
 that for all $k\in{\mathbb N}$, $\{f_k\}$ is uniformly bounded in
${{\rm TL_\L^\lambda}(\Real_+^{n+1})}$.
\end{proof}

\begin{proof}[Proof of part (1) of Theorem~\ref{th1.1}]
 Recall that the condition $V\in B_n$   implies $V\in B_{q_0}$ for some $q_0>n/2$.
From
Lemmas~\ref{le2.2} and  \ref{le3.8}, we see  that $u(x,t)=e^{-t{\L}}f(x)\in C^1({\mathbb R}^{n+1}_+)$.
It will be enough to finish the proof if we have proved
\begin{equation}\label{e3.16}
\|u\|_{\rm TL_\L^\lambda(\mathbb{R}^{n+1}_{+})}\leq C\|f\|_{L^{2,\lambda}(\mathbb{R}^{n})}.
\end{equation}
To prove (\ref{e3.16}), by a similar argument in \cite{DGMTZ}, we can easily prove that for every $f\in L^{2,\lambda}(\mathbb{R}^{n})$,
the term $|\partial_{t}e^{-t\mathcal{L}}(f)(x)|^{2}$ has the following estimate(see \cite[Theorem 2]{DGMTZ}):
\begin{equation*}
\sup_{x_B, r_B}
 r_B^{-\lambda}\int_0^{r_B^2}\int_{B(x_B, r_B)} |  \partial_t  e^{-t \L}   f(x)|^2 {dxdt }\leq\|f\|_{L^{2,\lambda}(\mathbb{R}^{n})}.
\end{equation*}
So, we only need to estimate the term $|\nabla_{x}e^{-t\mathcal{L}}(f)(x)|^{2}$. In fact
\begin{eqnarray*}
\left(\frac{1}{r_{B}^{\lambda}}\int_{0}^{r_{B}^{2}}\int_{B}|\nabla_{x} e^{-t\mathcal{L}}f(x)|^{2}dxdt\right)^{\frac{1}{2}}&\leq&
\sum\limits_{k=0}\limits^{\infty}\frac{1}{r_{B}^{\lambda/2}}\left(\int_{0}^{r_{B}^{2}}\int_{B}|\nabla_{x} e^{-t\mathcal{L}}f_{k}(x)|^{2}dxdt\right)^{\frac{1}{2}}\\
&=:&\sum\limits_{k=0}\limits^{\infty}J_{k},
\end{eqnarray*}
where $f_{0}=f\chi_{2B}$ and $f_{k}=f\chi_{2^{k+1}B\backslash2^{k}B}$ for $k\in\mathbb{N}^{+}$. For $J_0$, since the Riesz transform $\nabla \L^{-1/2}$ is bounded on $L^2(\RR)$, by \eqref{eqsquare} and the commutative property of $e^{-\L}$ and $\L^{-{1\over 2}}$,   we have
\begin{align*}
J_0^2=\frac 1{r_B^\lambda} \int_0^{r_B^2}\int_B \abs{\nabla_{x}e^{-t{ \L}}f_0(x) }^2{dxdt}
 &\le \frac 1{r_B^\lambda}\int_0^{r_B^2}\int_{\Real^n} \abs{\nabla_{x}\L^{-{1/2}}\L^{{1/2}}e^{-t {\L}}f_0(x) }^2{dxdt}\\
&\le C \frac 1{r_B^\lambda}\int_0^{\infty}\int_{\RR} \abs{\partial_t e^{-t{ \L}}\left(\L^{-1/2}f_0\right)(x) }^2{dxdt}\\
&\le C\frac 1{r_B^\lambda}\norm{\L^{1/2}\L^{-1/2}f_0}_{L^2(\Real^n)}^2=C\frac 1{r_B^\lambda}\int_{2B}\abs{f(x)}^2dx\\
&\le C\|f\|_{L^{2,\lambda}(\mathbb{R}^{n})}^2.
\end{align*} When $k\ge 1,$
for any $x\in B$ and $k\in\mathbb{N}^{+}$, we apply (\ref{e3.11}) to obtain
\begin{eqnarray*}
|\nabla_{x}e^{-t\mathcal{L}}f_{k}(x)|&\leq& C\int_{2^{k+1}B/2^{k}B}t^{-(n+1)/2}e^{-\frac{|x-y|^{2}}{ct}}|f(y)|dy\\
&\leq& C\int_{2^{k+1}B/2^{k}B}|x-y|^{-(n+1)}|f(y)|dy\\
&\leq& C \frac{1}{(2^{k}r_{B})^{n+1}}\int_{2^{k+1}B}|f(y)|dy\\
&\leq& C \frac{1}{(2^{k}r_{B})^{1+\frac{n-\lambda}{2}}}\|f\|_{L^{2,\lambda}(\mathbb{R}^{n})},
\end{eqnarray*}
which yields
\begin{equation*}
|J_{k}|\leq C2^{-k(1+\frac{n-\lambda}{2})}\|f\|_{L^{2,\lambda}(\mathbb{R}^{n})}.
\end{equation*}
Hence $\sum\limits_{k=0}\limits^{\infty}|J_{k}|\leq C\|f\|_{L^{2,\lambda}(\mathbb{R}^{n})}$, and then $\|u\|_{\rm TL_\L^\lambda(\mathbb{R}^{n+1})}\leq C\|f\|_{L^{2,\lambda}(\mathbb{R}^{n})}$.
\end{proof}

 \begin{proof}[Proof of part (2) of Theorem~\ref{th1.1}]  To prove it, we will use the argument as in \cite{DYZ,FJN,JX} and apply the
 key Lemma \ref{le3.5}. Suppose $u\in \rm TL_\L^\lambda(\mathbb{R}^{n+1}_{+})$, our aim is to find a function $f\in L^{2,\lambda}(\mathbb{R}^{n})$ such that
 \begin{equation*}
 u(x,t)=e^{-t\mathcal{L}}f(x),~~~~~\text{for each}~~~~~(x,t)\in\mathbb{R}^{n+1}_{+}
 \end{equation*}
 To do this, for every $k\in \mathbb{N}^{+}$, we write $f_{k}=u(x,{1/ k}).$ 
  By  Lemma \ref{le3.5}, we obtain
 \begin{equation*}
 \int_{B(0,2^{j})}|f_{k}(x)|^{2}dx\leq C2^{j\lambda}\|u\|_{\rm TL_\L^\lambda(\mathbb{R}^{n+1}_{+})}^{2}.
 \end{equation*}
 This means that the sequence $\{f_{k}\}_{k=1}^{\infty}$ is bounded in $L^{2}(B(0,2^{j}))$. So by passing to a subsequence, the
 sequence $\{f_k\}$ converges weakly to a function $g_{j}\in L^{2}(B(0,2^{j}))$. Then, for $i>j$, we can get $g_i(x)=g_j(x),$ for almost everywhere $x\in B(0, 2^j).$ Next, we define a function $f(x)$ by
 \begin{equation*}
 f(x)=g_{j}(x),\ \ \ \  ~~~~~~~~~~\text{if} ~~~~~~~~x\in B(0,2^{j}), j=1,2,3\cdots.
 \end{equation*}
 It is easy to see that $f$ is well defined on $\mathbb{R}^{n}=\bigcup_{j=1}^{\infty}B(0,2^{j})$ and (after passing to a subsequence) $f_k\rightarrow f$ in $L^2$ on every ball of $\mathbb{R}^n$. It is also easy to check that for any
 open ball $B\subset\mathbb{R}^{n}$, we have
 \begin{equation*}
 \int_{B}|f(x)|^{2}dx\leq Cr_{B}^{\lambda}\|u\|_{\rm TL_\L^\lambda(\mathbb{R}^{n+1}_{+})}^{2},
 \end{equation*}
 which implies
 \begin{equation*}
 \|f\|_{L^{2,\lambda}(\mathbb{R}^{n})}\leq C\|u\|_{\rm TL_\L^\lambda(\mathbb{R}^{n+1}_{+})}.
 \end{equation*}
 Finally, we will show that $u(x,t)=e^{-t\mathcal{L}}f(x)$. Since $u(x,\cdot)$ is continuous on $\mathbb{R}_{+}$, we have
 $u(x,t)=\lim\limits_{k\rightarrow+\infty}u(x,t+\frac{1}{k})$. Then we have
 $u(x,t)=\lim\limits_{k\rightarrow+\infty}e^{-t\mathcal{L}}(u(\cdot,\frac{1}{k}))(x)$. It reduces to show
 \begin{equation*}
 \lim\limits_{k\rightarrow+\infty}e^{-t\mathcal{L}}(u(\cdot,\frac{1}{k}))(x)=e^{-t\mathcal{L}}f(x).
 \end{equation*}
 Indeed, we recall that $\mathcal{H}_{t}(x,t)$ is the kernel of $e^{-t\mathcal{L}}$. Then for any $l\in\mathbb{N}$, we write
 \begin{equation*}
 e^{-t\mathcal{L}}(u(\cdot,\frac{1}{k}))(x)=\int_{B(x,2^{l}t)}\mathcal{H}_{t}(x,y)f_{k}(y)dy+
 \int_{B(x,2^{l}t)^{c}}\mathcal{H}_{t}(x,y)f_{k}(y)dy
 \end{equation*}
 By Lemma \ref{le2.2} and the H\"{o}lder inequality, we have
 \begin{eqnarray*}
 \left|\int_{(B(x,2^{l}t))^{c}}\mathcal{H}_{t}(x,y)f_{k}(y)dy\right|
 &\leq&C\sum\limits_{i=l}\limits^{\infty}\int_{B(x,2^{i+1}t)/B(x,2^{i}t)}t^{-\frac{n}{2}}e^{-\frac{|x-y|^{2}}{ct}}|f_{k}(y)|dy\\
 &\leq&C\sum\limits_{i=l}\limits^{\infty}(2^{i}t)^{-n}
 \int_{B(x,2^{i+1}t)}|f_{k}(y)|dy\\
 &\leq& C\sum\limits_{i=l}\limits^{\infty}(2^{i}t)^{\frac{\lambda-n}{2}}\|f_{k}\|_{L^{2,\lambda}(\mathbb{R}^{n})}\\
 &\leq& C2^{-l\frac{n-\lambda}{2}}t^{\frac{\lambda-n}{2}}\|f_{k}\|_{L^{2,\lambda}(\mathbb{R}^{n})}.
 \end{eqnarray*}
By Lemma \ref{le3.5}, we have that $\|f_{k}\|_{L^{2,\lambda}(\mathbb{R}^{n})}\leq C\|u\|_{\rm TL_\L^\lambda(\mathbb{R}^{n+1}_{+})}$ for some constant $C>0$
 independent of $k$. Since $\lambda\in (0,n)$, we have
 \begin{equation*}
\limsup\limits_{l\rightarrow+\infty}\limsup\limits_{k\rightarrow+\infty}\abs{\int_{B(x,2^{l}t)^c}\mathcal{H}_{t}(x,y)f_{k}(y)dy}\le
 \lim\limits_{l\rightarrow+\infty}\left(C2^{-l(\frac{n-\lambda}{2})}t^{\frac{\lambda-n}{2}}\|u\|_{\rm TL_\L^\lambda(\mathbb{R}_{+}^{n+1})}\right)=0.
 \end{equation*}
 Therefore,
 \begin{equation*}
 \lim\limits_{k\rightarrow+\infty}e^{-t\mathcal{L}}(u(\cdot,k^{-1}))(x)=\lim\limits_{k\rightarrow+\infty}
 \lim\limits_{l\rightarrow+\infty}\int_{B(x,2^{l}t)}\mathcal{H}_{t}(x,y)f_{k}(y)dy=e^{-t\mathcal{L}}f(x).
 \end{equation*}
 We have showed that $u(x,t)=e^{-t\mathcal{L}}f(x)$. The proof of Theorem \ref{th1.1} is completed.
\end{proof}

\bigskip

{\bf Acknowledgments.}  This research was partly supported by the Natural Science Foundation of Zhejiang Province (Grant No. LY18A010006  and No. LQ18A010005) and the National Natural Science Foundation of China (Grant No. 11401525 and No. 11801518).

 The authors would like to thank the referee's professional comments which improved our results in this paper.

 \bigskip




\begin{thebibliography}{10}




 \bibitem {DKP} M. Dindos, C. Kenig and J.  Pipher,
  BMO solvability and the $A_{\infty}$ condition for elliptic operators. \textit{J. Geom. Anal. } \textbf{21} (2011), 78--95.




\bibitem{DXY}X. Duong, J. Xiao and L. Yan, Old and new Morrey spaces with heat kernel bounds.
\textit{J. Fourier Anal. Appl.} \textbf{13} (2007), 87--111.

\bibitem   {DY1}  X. Duong and L. Yan, New function spaces of
{\rm BMO} type, the John-Nirenberg inequality, interpolation and applications,
{\it  Comm. Pure Appl. Math.} {\bf 58} (2005), 1375--1420.


\bibitem {DY2} X. Duong and L. Yan, Duality of Hardy and BMO spaces
associated with operators with heat kernel bounds. {\it J. Amer. Math. Soc.}
{\bf 18}(2005), 943--973.

\bibitem {DYZ}X. Duong, L. Yan and C. Zhang,
On characterization of Poisson integrals of Schr\"odinger operators with BMO traces. {\it J. Funct. Anal.} {\bf 266}(2014),  2053-2085.

\bibitem{DGMTZ} J. Dziuba\'nski, G. Garrig\'os, T. Mart\'inez, J.  Torrea and J. Zienkiewicz,
{$BMO$ spaces related to Schr\"odinger operators with potentials satisfying a reverse H\"older inequality}.
\textit{Math. Z.}
\textbf{249} (2005), 329--356.


\bibitem{DZ2002} J. Dziuba\'nski and J. Zienkiewicz,
{$H^p$ spaces for Schr\"odinger operators}, in: Fourier Analysis and Related Topics \textbf{56}, Banach Center Publ., Inst. Math., Polish Acad. Sci.,
Warszawa, 2002, 45--53.


\bibitem{FJN}  E. Fabes, R. Johnson and U. Neri,
Spaces of harmonic functions representable by Poisson integrals of functions in BMO and $L_{p, \lambda}$.
\textit{ Indiana Univ. Math. J.}
\textbf{ 25} (1976), 159--170.


\bibitem{FN1} E. Fabes and U. Neri,
 Characterization of temperatures with initial data in BMO. \textit{ Duke Math. J. }
\textbf{42} (1975),  725-734.

\bibitem{FN} E. Fabes and U. Neri,   Dirichlet problem in Lipschitz domains with BMO data.
\textit{Proc. Amer. Math. Soc.}
\textbf{78} (1980), 33--39.



\bibitem{FS} C. Fefferman and E. Stein,   $H^p$ spaces of
 several variables. {\it Acta
Math.} {\bf 129} (1972), 137--195.



\bibitem{GJ} W. Gao and Y. Jiang, $L_p$   estimate for parabolic Schr\"odinger operator with certain potentials. \textit{J. Math. Anal. Appl.}  \textbf{310}  (2005), 128--143.

\bibitem{Ge} F. Gehring, The $L^p$-integrability of the partial derivatives of a quasiconformal mapping.
{\it Acta Math.},  {\bf 130} (1973), 265--277.





\bibitem  {HLMMY} S. Hofmann, G. Lu, D. Mitrea, M. Mitrea and L. Yan,
 Hardy spaces associated to non-negative
self-adjoint  operators satisfying Davies-Gaffney estimates.
{\it Memoirs of the Amer. Math. Soc.}, {\bf 214} (2011), no. 1007.


\bibitem{HMM} S. Hofmann, J.  Martell and S. Mayboroda,  Layer potentials and boundary value problems for elliptic equations with complex $L^\infty$
   coefficients satisfying the small Carleson measure norm condition, \textit{Adv. Math.} \textbf{270} (2015), 480--564.



\bibitem{JX} R. Jiang, J. Xiao and D. Yang, Towards spaces of harmonic functions with traces in square Campanato spaces and their scaling invariants, \textit{Anal. Appl. (Singap.)} \textbf{14} (2016), 679--703.




\bibitem{MSTZ} T.  Ma, P.  Stinga, J.  Torrea and C. Zhang,
{Regularity properties of Schr\"odinger operators},
\textit{J. Math. Anal. Appl.}
\textbf{388} (2012), 817--837.


\bibitem{MSTZ2} T.  Ma, P.  Stinga, J.  Torrea and C. Zhang,
{Regularity estimates in H\"older spaces for Schr\"odinger operators via a T1  theorem},
\textit{Ann. Mat. Pura Appl.}
\textbf{193} (2014), 561--589.

\bibitem{Morrey} C. Morrey,
{On the solutions of quasi-linear elliptic partial differential equations},
\textit{Trans. Amer. Math. Soc.}
\textbf{43} (1938),  126--166.


\bibitem{Shen} Z. Shen,
{$L^p$ estimates for Schr\"odinger operators with certain potentials}.
\textit{Ann. Inst. Fourier (Grenoble)}
\textbf{45} (1995), 513--546.


\bibitem{Shen1999} Z. Shen,   On fundamental solution of generalized Schr\"odinger
operators.  {\it J. Funct. Anal.}   {\bf 167}
(1999),  521--564.

\bibitem{Song} L. Song, X. Tian and L. Yan, On characterization of Poisson integrals of Schr\"odinger operators with Morry traces, {\it Acta Math. Sin. (Engl. Ser.)}   {\bf 34} (2018),  787--800.

\bibitem{St1970} E. Stein,
\textit{Topics in Harmonic Analysis Related to the Littlewood-Paley Theory},
{Annals of Mathematics Studies} \textbf{63},
Princeton Univ. Press,
Princeton, NJ, 1970.




\bibitem{SW} E.  Stein and G. Weiss,
\textit{Introduction to Fourier Analysis on Euclidean spaces},
Princeton Univ. Press,
Princeton, NJ, 1970.



\bibitem{TH} L. Tang and J. Han, $L_p$ boundedness for parabolic Schr\"odinger type operators with certain nonnegative potentials,
{\it Forum Math.} {\bf 23} (2011), 161--179.


\bibitem{WY} L. Wu and L, Yan, Heat kernels, upper bounds and Hardy spaces associated to the generalized Schr\"odinger operators, {\it J. Funct.
Anal.} {\bf 270} (2016), 3709-3749.

\bibitem{YZ} M. Yang and C. Zhang, Characterization of temperatures associated to Schr\"odinger operators with initial data in BMO spaces, {\it arXiv:1710.01160}.

 \bibitem{YSY} W. Yuan, W. Sickel and D. Yang,
\textit{Morrey and Campanato meet Besov, Lizorkin and Triebel},
Lecture Notes in Mathematics, 2005. Springer-Verlag, Berlin, 2010.

  .



\end{thebibliography}
\end{document}